\documentclass[hha]{article}
\usepackage{amsmath}
\usepackage{amssymb}
\usepackage{amsthm}
\usepackage[mathscr]{eucal}
\usepackage[all,knot,cmtip]{xy}
\xyoption{arc}
\usepackage{soul}
\usepackage{graphicx}
\usepackage{color}

\numberwithin{equation}{section}
\newtheorem{thm}{Theorem}[section]

\newtheorem{prop}[thm]{Proposition}

\newcommand{ \WL }{ \mathbb{W}(\mathbb{L}) }
\newcommand{\WLCat}{\mathbb{W}(\mathbb{L}) \mathchar`-\mathbf{Cat}}
\newcommand{\LCat}{\mathbb{L}\mathchar`-\mathbf{Cat}}

\theoremstyle{definition}
\newtheorem{definition}[thm]{Definition}
\newtheorem{rem}[thm]{Remark}

\makeatletter
\newcommand{\prc}{\mathbin{\mathpalette\prc@inner\relax}}
\newcommand{\prc@inner}[2]{%
  \vbox{\offinterlineskip\m@th
    \ialign{%
      ##\cr
      \hidewidth\raisebox{-1.5\height}[0pt][0pt]{$#1.$}\hidewidth\cr
      $#1-$\cr
    }%
  }%
}
\makeatother

\begin{document}
\title{On the Witt vectors in the Lawvere quantale}
\author{Ryo Horiuchi}
\date{}

\maketitle

\begin{abstract}
In this note, we investigate the Witt vectors in the min-plus algebra of extended non-negative real numbers, and consider categories enriched over them viewed as a monoidal category.
\end{abstract}

\section{Introduction}
It is well known that topological Hochschild homology relates to Witt vectors in a profound way (\cite{HM} and \cite{BHM}), and this fact has been playing a pivotal role around the algebraic $K$-theory over the past several decades.

With topological Hochschild homology as one of its basic examples, the theory of cyclotomic spectra has been established (\cite{HM}, \cite{BM}, and \cite{NS}) based on a plethora of work in equivariant stable homotopy theory.
There have been numerous contributions to and applications of the theory of cyclotomic spectra.
The theory of $\lambda$-rings, which are defined in terms of Witt vectors, has been playing a great role as well.
In particular, it is a central object in \cite{B1} and \cite{B2}.
Moreover, using the framework of plethystic algebra (\cite{TW}, \cite{BW}, and \cite{B3}), the Witt vectors of commutative rigs have also been studied in \cite{B3} and \cite{BG}.

Recently, the notion of categorical spectra has been established as an extension of spectra \cite{S}.
There, categorical spectra are defined as the stabilisations of $(\infty, \infty)$-categories instead of $(\infty, 0)$-categories.
Moreover, higher algebras accordingly have been extended to the categorical one \cite{Masuda}.
Therefore, it has become possible to analyse a broader class of algebras in terms of homotopy theory.
In particular, while rigs of characteristic one are not objects of higher algebra, they can be studied within the extended higher algebra.
It might be natural to expect that the relationship between topological Hochschild homology and Witt vectors as well as that between cyclotomic spectra and $\lambda$-rings would lift to this extended framework.

In this note, we explore how Witt vectors behave for a specific rig of characteristic one.
More precisely, we investigate the Witt vectors in the min-plus algebra $[0, \infty]$ and relate them to metric spaces in the sense of Lawvere \cite{L}.
We hope this work will contribute to the theory of cyclotomic categorical spectra, which, to the best of our knowledge, has not yet been explored much.

\section{Witt vectors and Lawvere metric spaces}
\subsection{Witt vectors}

In this section, we briefly recall the basics of the theory of Witt vectors, following \cite{B3} and \cite{BW}, as preparation for later sections.

Let $\Lambda$ denote the set of symmetric functions with coefficients in $\mathbb{N}$.
As is shown in \cite{B3}, $\Lambda$ is a (commutative and co-commutative) birig, namely a commutative rig object in $\mathsf{CRig}^{op}$, where $\mathsf{CRig}$ denotes the category of commutative rigs and rig homomorphisms, and $\mathsf{CRig}^{op}$ its opposite category.
In the usual way, $\Lambda$ is a commutative rig, and with the following structures, it gives rise to a birig.
See \cite{B3} for more detail:

\[
\Delta^{+} : \Lambda \to \Lambda \otimes_{\mathbb{N}} \Lambda,
\]
\[
\Delta^{+}(f(x_1, x_2, \dots)) =
  f(x_1 \otimes 1, 1 \otimes x_1, 
    x_2 \otimes 1, 1 \otimes x_2, 
    \dots, 
    x_i \otimes 1, 1 \otimes x_i, 
    \dots)
\]

\[
\varepsilon^{+} : \Lambda \to \mathbb{N},
\qquad
\varepsilon^{+}(f(x_1, x_2, \dots)) = f(0,0,0,\dots)
\]

\[
\Delta^{\times} : \Lambda \to \Lambda \otimes_{\mathbb{N}} \Lambda,
\]
\[
\Delta^{\times}(f(x_1, x_2, \dots)) =
  f(x_1 \otimes x_1, x_1 \otimes x_2, 
    \dots, 
    x_i \otimes x_j, 
    \dots)
\]

\[
\varepsilon^{\times} : \Lambda \to \mathbb{N},
\qquad
\varepsilon^{\times}(f(x_1, x_2, \dots)) = f(1,0,0,\dots),
\]
where $\Delta^{+}$, $\varepsilon^{+}$, $\Delta^{\times}$, and $\varepsilon^{\times}$ will be called co-addition, co-additive unit, co-multiplication, and co-multiplicative unit, respectively.

Let $(S, +_S, \cdot_S)$ be a commutative rig and $\mathbb{W}(S)$ denote the set $\operatorname{Hom}_{\mathbf{CRig}}(\Lambda, S)$ of rig homomorphisms.
We give $\mathbb{W}(S)$ the following rig structure:

Let $f, g\in\mathbb{W}(S)$. 
The addition $f\oplus_{\mathbb{W}(S)}g\in\mathbb{W}(S)$ is given by
\[(f\oplus_{\mathbb{W}(S)}g)(\varphi)=\cdot_S\circ(f\otimes g)\circ\Delta^{+}(\varphi)\]
for any $\varphi\in\Lambda$.
The additive unit is induced by $\varepsilon^+$.
Similarly, the multiplication $f\otimes_{\mathbb{W}(S)}g\in\mathbb{W}(S)$ is given by
\[(f\otimes_{\mathbb{W}(S)}g)(\varphi)=\cdot_S\circ(f\otimes g)\circ\Delta^{\times}(\varphi)\]
for any $\varphi\in\Lambda$ and the multiplicative unit is induced by $\varepsilon^\times$.

\begin{thm}[\cite{B3}]
This gives rise to a functor $\mathbb{W}:\mathbf{CRig}\to \mathbf{CRig}$.
\end{thm}
This theorem has been proven by \cite{TW} for commutative rings.

The birig $\Lambda$ also has an important structure called composition $\mathbb{N}$-algebra.
As is shown in \cite{B3}, the category of commutative and cocommutative birigs admits a monoidal structure where the monoidal product is denoted by $\odot$ and the unit object is the polynomial birig $\mathbb{N}[X]$.

\begin{definition}[\cite{B3}]A monoid object with respect to this monoidal structure is called a composition $\mathbb{N}$-algebra.
\end{definition}

\begin{thm}[\cite{B3}]The birig $\Lambda$ is a composition $\mathbb{N}$-algebra.
\end{thm}

The monoid operation $\circ:\Lambda\odot\Lambda\to\Lambda$ is the one known as plethysm or composition in \cite{Macdonald}.
The unit map $\mathbb{N}[X]\to\Lambda$ is induced by $X\mapsto m_{(1)}$.

By this structure, the functor $\mathbb{W}:\mathbf{CRig}\to \mathbf{CRig}$ gives rise to a comonad.
Namely, the structure map, for any commutative rig $S$, $\mathbb{W}(S)\to\mathbb{W}(\mathbb{W}(S))$ is given by $f\mapsto \tilde{f}$, which is given by $\tilde{f}(\varphi)=f(\mathchar`-\circ\varphi)\in\mathbb{W}(S)$ for any $\varphi\in\Lambda$.

Restricting to commutative rings and abusing notation, $\mathbb{W}:\mathbf{CRing}\to \mathbf{CRing}$ is also a comonad, where $\mathbf{CRing}$ denotes the category of commutative rings and ring homomorphisms.

\begin{definition}[\cite{TW}]A coalgebra over the comonad $\mathbb{W}:\mathbf{CRing}\to \mathbf{CRing}$ is called a $\lambda$-ring.
\end{definition}

In some literature, the $\lambda$-ring structure of a $\lambda$-ring $R$ is defined in terms of endomaps on $R$ indexed by $\mathbb{N}$.
Such a definition and the one given above are connected via the elementary symmetric functions.
See \cite{TW} for more detail.
$\lambda$-rings have been playing a prominent role around the algebraic $K$-theory.
It is also crucial in Borger's absolute geometry \cite{B1} and \cite{B2}.
It would be fair to expect coalgebras over the comonad $\mathbb{W}:\mathbf{CRig}\to \mathbf{CRig}$ could also play a role in math.
However we do not investigate such coalgebras in this paper, except for $\WL$ defined later.

\begin{rem}For any commutative ring $R$, the commutative rig $\mathbb{W}(R)$ is canonically isomorphic to the ring of the  usual (big) Witt vectors in $R$.
Also, as is well known, there is a ring isomorphism
\[\mathbb{W}(R)\to 1+tR[[t]], \varphi\mapsto\sum_{n\geq0}\varphi(e_n)t^n,\]
where $e_0:=1$ and $e_n$ is the $n$-th elementary symmetric function, which plays an important role in the theory of Witt vectors in rings.

However, since the evident map $\mathbb{N}[e_1, e_2, \dots]\to\Lambda$ is not a bijection, the ring isomorphism above does not extend to rig isomorphism in general.
More precisely, we have a map, for a commutative rig $S$, 
\[\operatorname{Hom}_{\mathbf{CRig}}(\mathbb{N}[e_1, e_2, \dots], S)\to 1+tS[[t]], \varphi\mapsto\sum_{n\geq0}\varphi(e_n)t^n.\]
However, we know no natural rig structure on $\operatorname{Hom}_{\mathbf{CRig}}(\mathbb{N}[e_1, e_2, \dots], S)$ in general.
\end{rem}

\subsection{Lawvere metric spaces}
We let $[0, \infty]$ denote the set of non-negative real numbers together with $\infty$, and view it as a linearly ordered set with the order $\leq_{\mathbb{L}}$ defined by
\[x\leq_{\mathbb{L}}y :\Leftrightarrow y \leq x,\]
where $\leq$ is the usual order on $[0, \infty]$.
We also equip $[0, \infty]$ with a monoid structure by the usual addition, denoted here by $\otimes_{\mathbb{L}}$ with the unit $0$.
These two structures are compatible in such a way that they give rise to a monoidal poset, hence a monoidal category, which we denote by $\mathbb{L}$.
It is well known that $\mathbb{L}$ is furthermore a quantale.
We refer to it as the Lawvere quantale.
While we do not discuss it in this paper, quantales are algebraic structures often studied in logic and related fields.
See \cite{AV} for example.

Note that the order structure $\leq_{\mathbb{L}}$ induces a binary operation $\oplus_{\mathbb{L}}$ on $[0, \infty]$ defined by
\[x\oplus_{\mathbb{L}}y:=x\vee y,\]
where $\vee$ denotes the join with respect to $\leq_{\mathbb{L}}$.
In other words, $\oplus_{\mathbb{L}}$ coincides with the minimum operation, $\operatorname{min}$, with respect to the usual order.

The Lawvere quantale $\mathbb{L}$ also determines the commutative rig whose underlying set is $[0, \infty]$, with addition given by $\operatorname{min}$ and multiplication given by $\otimes_{\mathbb{L}}$.
Abusing both terminology and notation, we may also refer to this commutative rig as the Lawvere quantale and denote it by $\mathbb{L}$.
Note that this rig is of characteristic one, namely $0\oplus_{\mathbb{L}}0=0$, which is an object in idempotent analysis.
Note also that $\leq_{\mathbb{L}}$ is recovered by $\oplus_{\mathbb{L}}$:
\[x\leq_{\mathbb{L}}y\Leftrightarrow x\oplus_{\mathbb{L}}z=y \ \text{for some} \ z\in[0, \infty].\]
In this sense, the order and the addition of $\mathbb{L}$ are equivalent.

It may be trivial but important that, for any ring $R$ and $x, y\in R$, we have $y-x\in R$ and $x+(y-x)=y$.
That is, the induced order on $R$ is trivial.

Let $\mathbb{L}\mathchar`-\mathbf{Cat}$ denote the category of (small) categories enriched over $\mathbb{L}$.
Let $X\in\LCat$.
Then, we have a hom-map $d:X\times X\to \mathbb{L}$, where, by abusing notation, we denote $X$ by objects in $X$ and $\mathbb{L}$ by objects in $\mathbb{L}$, which satisfies
\[0\leq_{\mathbb{L}}d(x, x)\]
\[d(x, x')\otimes_{\mathbb{L}}d(x', x'')\leq_{\mathbb{L}}d(x, x''),\]
for any $x, x', x''\in X$.
The first one may be called the identity axiom, and the second one the composition axiom.

In the usual notation, we have $0=d(x, x)$ and $d(x, x')+d(x', x'')\geq d(x, x'')$.
So we can regard $X$ as a kind of metric space.
This beautiful idea is due to Lawvere \cite{L}.
Note that the distance is not necessarily symmetric and there may exist some points $x, x'\in X$ which are not the same but $d(x, x')=0$.
Viewing categories enriched over quantales as a kind of space has long been an active area of research.
See \cite{HST} for more detail.

\begin{rem}
$\mathbb{L}$ itself can be viewed as an $\mathbb{L}$-category with the following distance:
\[d:\mathbb{L}\times\mathbb{L}\to\mathbb{L}, d(x, y)=\max\{y-x, 0\},\]
where $-$ is the extended subtraction and $\max$ is also the usual one.
We may write  $\max\{y-x, 0\}$ as $y\prc x$.
It may be obvious that this truncated subtraction $\prc$ defines the right adjoint to $\otimes_{\mathbb{L}}$, so $\mathbb{L}$ is a closed monoidal category.
\end{rem}

\section{$\WL$-categories}
In this section, viewing $\mathbb{L}$ as a rig, we consider the Witt vectors $\WL$.
We give it a monoidal category structure induced by the structures of $\Lambda$ and $\mathbb{L}$, and consider categories enriched over it.
We would like to think of such categories as a generalization of Lawvere metric spaces.

\subsection{Basic properties of $\WL$-categories}
\begin{definition}For any $n\in\mathbb{N}_{>0}$, we define $\mathbb{P}_{n}$ to be the set of partitions of $n$.
We define $\mathbb{P}$ to be the union of $\mathbb{P}_{n}$ for all $n\in\mathbb{N}_{>0}$.
\end{definition}

Let $f, g\in\mathbb{W}(\mathbb{L})$. 
We define an order $\leq_{\mathbb{W}(\mathbb{L})}$ as follows:
\[f\leq_{\mathbb{W}(\mathbb{L})} g:\Leftrightarrow \forall\lambda\in\mathbb{P}, f(m_{\lambda})\leq_{\mathbb{L}} g(m_{\lambda})\]

We denote the monoidal product on $\mathbb{W}(\mathbb{L})$ by $\otimes_{\mathbb{W}(\mathbb{L})}$ that is the multiplication of the rig $\WL$.
Thus, by definition,
\[(f\otimes_{\mathbb{W}(\mathbb{L})}g)(m_{\lambda})=\min_{\mu, \mu'}\{f(m_{\mu})+g(m_{\mu'})\},\]
where $\mu$ and $\mu'$ are the partitions that appear in the unique expression $\Delta^{\times}(m_{\lambda})=\sum_{\mu, \mu'}c_{\mu, \mu'}^{\lambda}m_{\mu}\otimes m_{\mu'}$.

Let $\overline{0}\in\mathbb{W}(\mathbb{L})$ be the element defined by $\overline{0}(m_{(n)})=0$ for all $(n)\in\mathbb{P}$, where $(n)$ denotes the partition $(n, 0, 0, \dots)$ of $n$, and $\overline{0}(m_{\lambda})=\infty$ for all other $\lambda\in\mathbb{P}$, that is the multiplicative unit of the rig $\WL$.

\begin{prop}
With these structures, $\mathbb{W}(\mathbb{L})$ is a monoidal poset.
\end{prop}
\begin{proof}It is clear that $\mathbb{W}(\mathbb{L})$ is a poset with $\leq_{\WL}$.
Since $\otimes_{\WL}$ and $\overline{0}$ are parts of the rig structure of $\WL$, the unit axiom and the associativity axiom hold.
We only need to check that the bifunctoriality of $\mathchar`-\otimes_{\WL}\mathchar`-:\WL\times\WL\to\WL$.
Since the multiplication is commutative, it is enough to show the functoriality on the first component.

Let $f, f'\in\WL$ with $f\leq_{\WL}f'$.
Then by definition for any $\nu\in\mathbb{P}$ we have $f(m_{\nu})\geq f'(m_{\nu})$.
For any $g\in\WL$ and $\lambda\in\mathbb{P}$, we have
\begin{eqnarray*}
(f\otimes_{\mathbb{W}(\mathbb{L})}g)(m_{\lambda})&=&\min_{\mu, \mu'}\{f(m_{\mu})+g(m_{\mu'})\} \\
&\geq&\min_{\mu, \mu'}\{f'(m_{\mu})+g(m_{\mu'})\}\\
&=&(f'\otimes_{\mathbb{W}(\mathbb{L})}g)(m_{\lambda}),
\end{eqnarray*}
where $\mu$ and $\mu'$ are the partitions that appear in the unique expression $\Delta^{\times}(m_{\lambda})=\sum_{\mu, \mu'}c_{\mu, \mu'}^{\lambda}m_{\mu}\otimes m_{\mu'}$ with $c_{\mu, \mu'}^{\lambda}\in\mathbb{N}$.

Thus we have $f\otimes_{\WL}g\leq_{\WL} f'\otimes_{\WL}g$.
\end{proof}

\begin{rem}
Note that $\{m_{\lambda}\}_{\lambda\in\mathbb{P}}$ forms an $\mathbb{N}$-module basis for $\Lambda$, and that elements of $\WL$ are rig homomorphisms.
There may exist a more natural way to endow $\WL$ with a partial order which makes it a complete partially ordered set.
\end{rem}

Let $\WL\mathchar`-\mathbf{Cat}$ denote the category of (small) categories enriched over the monoidal poset $\WL$.
Let $X\in\WL\mathchar`-\mathbf{Cat}$.
Then we have a hom-map $d:X\times X\to \WL$, where, by abusing notation, we denote $X$ by objects in $X$ and $\WL$ by objects in $\WL$, which satisfies
\[\overline{0}\leq_{\WL}d(x, x)\]
\[d(x, x')\otimes_{\WL}d(x', x'')\leq_{\WL}d(x, x''),\]
for any $x, x', x''\in X$.

For any $\lambda\in\mathbb{P}$, we define for simplicity the map $d_{\lambda}:X\times X\to\mathbb{L}$ given by $d_{\lambda}(x, x')=d(x, x')(m_{\lambda})$ for any $x, x'\in X$.

Then, from the first inequality above, we obtain by the definition of $\overline{0}$
\[0=d_{(n)}(x, x)\]
for any $n\in\mathbb{N}_{>0}$, where $(n)$ denotes the partition $(n, 0, 0, \dots)$ of $n$.
Note that for any partition $\lambda\in\mathbb{P}$ not of the form $(n)$, we have only $\infty\geq d_{\lambda}(x, x)$, which is trivial.
From the second one above, for any $\lambda\in\mathbb{P}$ we have
\[d_{\lambda}(x, x'')\leq \min_{\mu, \mu'}\{d_{\mu}(x, x')+d_{\mu'}(x', x'')\},\]
where $\mu$ and $\mu'$ run over the sum $\Delta^{\times}(m_{\lambda})=\sum_{\mu, \mu'}c_{\mu, \mu'}^{\lambda}m_{\mu}\otimes m_{\mu'}$.

\begin{prop}For any $\mathbb{W}(\mathbb{L})$-categeory $(X, d)$ and any $n\in\mathbb{N}_{>0}$, $(X, d_{(n)})$ is an $\mathbb{L}$-category.
\end{prop}

\begin{proof}By the definition of $\mathbb{W}(\mathbb{L})$, for any $x\in X$we have
\[\overline{0}\leq_{\mathbb{W}(\mathbb{L})}d(x, x).\]
Since $\overline{0}(m_{(n)})=0$, we have $0=d_{(n)}(x, x)$ as we saw.
Also for any $x, y, z\in X$, by the definition of $\WL$, we have
\[d(x, z)\geq_{\WL} d(x, y)\otimes_{\WL}d(y, z)\]
Since $\Delta^{\times}(m_{(n)})=m_{(n)}\otimes m_{(n)}$, by the definition of $\leq_{\WL}$, we have 
\[d_{(n)}(x, y)\leq d_{(n)}(x, y)+d_{(n)}(y, z).\]
\end{proof}
So a $\WL$-category $(X, d)$ contains a sequence $\{(X, d_{(n)})\}$ of $\mathbb{L}$-categories, with a common underlying set, indexed by $\mathbb{N}_{>0}$, which may be viewed as discrete time.
One may think of it as the shape of $X$ changes moment by moment as time passes, while the particles remain the same.
Note that there may be some particles $x, x'\in X$ with $x\neq x'$ satisfy $d_{(n)}(x, x')=0$ at some phase $n$ and $d_{(n')}(x, x')\neq 0$ at another phase $n'$.

\begin{prop}For any $\mathbb{W}(\mathbb{L})$-categeory $(X, d)$ and any $\lambda\in\mathbb{P}$, $d_{\lambda}$ satisfies the triangle inequality.
\end{prop}
\begin{proof}We prove, for any $x, y, z\in X$ and any $\lambda=(\lambda_1, ..., \lambda_n)\in\mathbb{P}$, $d_{\lambda}(x, z)\leq d_{\lambda}(x, y)+d_{\lambda}(y, z)$.

By the definition of $\WL$-category, we have
\[d_{\lambda}(x, z)\leq (d(x, y)\otimes_{\WL}d(y, z))(m_{\lambda}).\]
Again by the definition of $\otimes_{\WL}$, we have 
\[(d(x, y)\otimes_{\WL}d(y, z))(m_{\lambda})=\min_{\mu, \mu'}\{d_{\mu}(x, y)+d_{\mu'}(y, z)\},\]
where $\mu$ and $\mu'$ are the partitions which appear in the sum 
\[\Delta^{\times}(m_{\lambda})=\sum_{\mu, \mu'}c_{\mu, \mu'}^{\lambda}m_{\mu}\otimes m_{\mu'}.\]
By definition, the term $(x_1\otimes x_1)^{\lambda_1}\cdots(x_n\otimes x_n)^{\lambda_n}=(x_1^{\lambda_1}\cdots x_n^{\lambda_n})\otimes(x_1^{\lambda_1}\cdots x_n^{\lambda_n})$ appears in $\Delta^{\times}(m_{\lambda})=m_{\lambda}(x_1\otimes x_1, x_1\otimes x_2, ..., x_i\otimes x_j, ...)$ at least once.
Therefore, we have
\[(d(x, y)\otimes_{\WL}d(y, z))(m_{\lambda})\leq d_{\lambda}(x, y)+d_{\lambda}(y, z).\]
\end{proof}

For a $\WL$-category $(X, d)$ and any $\lambda\in\mathbb{P}$ not of the form $(n)$ for any $n\in\mathbb{N}_{>0}$, $(X, d_{\lambda})$ is still a kind of metric space.
However in them there may exist some particle $x\in X$ for which the self-distance $d_{\lambda}(x, x)$ is not $0$.
Note that, in a partial metric space, a point can have non-zero distance between itself \cite{Matthew}.
Thus a $\WL$-category $(X, d)$ consists of a family $\{(X, d_{\lambda})\}_{\lambda\in\mathbb{P}}$ of space-like structures indexed by $\mathbb{P}$ with a common underlying set, whose components are related by the following inequalities mentioned above, namely for any  $x, x', x''\in X$ and $\lambda\in\mathbb{P}$,
\[d_{\lambda}(x, x'')\leq \min_{\mu, \mu'}\{d_{\mu}(x, x')+d_{\mu'}(x', x'')\},\]
where $\mu$ and $\mu'$ are the partitions appearing in $\Delta^{\times}(m_{\lambda})$.
Note that, by the definition of $\Delta^{\times}$, we have $\mu, \mu'\in\mathbb{P}_{|\lambda|}$ for such $\mu$ and $\mu'$, where $|\lambda|$ denotes $\lambda_1+\lambda_2+\cdots +\lambda_{n}$ with $\lambda=(\lambda_1, \lambda_2, \dots \lambda_{n})$.

Thus it may be natural to think about the complete symmetric functions $h_n=\sum_{\lambda\in\mathbb{P}_{n}}m_{\lambda}$ of degree $n$, for all $n$.

\begin{definition}For $(X, d)\in\WL$ and $n\in\mathbb{N}_{>0}$, we define $d_{n}$ as follows:
\[d_{n}:X\times X\to \mathbb{L}, (x, y)\mapsto d(x, y)(h_n).\]
\end{definition}

\begin{prop}\label{complete-metric}For any $\mathbb{W}(\mathbb{L})$-categeory $(X, d)$ and any $n\in\mathbb{N}$, $(X, d_{n})$ is an $\mathbb{L}$-category.
\end{prop}

\begin{proof}For any $x\in X$, by the definition of $d_{n}$, we have 
\[d_{n}(x, x)=\min_{|\lambda|=n}\{d_{\lambda}(x, x)\}\leq d_{(n)}(x, x)=0.\]

Let $x, y, z\in X$.
By definition, we have
\begin{eqnarray*}
d_{n}(x, z)&\leq&(d(x, y)\otimes_{\WL}d(y, z))(h_n) \\
&=&\min_{|\lambda|=n}\{(d(x, y)\otimes_{\WL}d(y, z))(m_{\lambda})\}\\
&=&\min_{|\lambda|=n}\{\min_{\mu, \mu'}\{d_{\mu}(x, y)+d_{\mu'}(y, z)\}.
\end{eqnarray*}
Here, $\mu$ and $\mu'$ are the partitions which appear in $\Delta^{\times}(m_{\lambda})$.
So in particular, $|\mu|=|\mu'|=n$.
Moreover, in $\min_{|\lambda|=n}\{\min_{\mu, \mu'}\{d_{\mu}(x, y)+d_{\mu'}(y, z)\}$, $\lambda$ runs over all partitions of size $n$.
Therefore, we have
\begin{eqnarray*}
(d(x, y)\otimes_{\WL}d(y, z))(h_n)&=&\min_{|\mu|=|\mu'|=n}\{d_{\mu}(x, y)+d_{\mu'}(y, z)\}\\
&=&\min_{|\mu|=n}\{d_{\mu}(x, y)\}+\min_{|\mu'|=n}\{d_{\mu'}(y, z)\}\\
&=&d_{n}(x, y)+d_{n}(y, z).
\end{eqnarray*}
\end{proof}

\begin{rem}Consider, for example, the Plancherel measure on $\mathbb{P}_{n}$ for every $n\in\mathbb{N}_{>0}$.
These give rise to a Poissonized Plancherel measure on $\mathbb{P}$ which is a standard tool for the theory of asymptotic analysis of the Young lattice.
See \cite{BO} for more detail.

Let $X\in\WLCat$.
Then we obtain the family $\{(X, d_{\lambda})\}_{\lambda\in\mathbb{P}}$.
So it might be fine to say X collects all possible ways in which a {\it space} may grow probabilistically.
Note that the probability that a randomly chosen $\lambda$ is of the form $(|\lambda|)$ tends to decrease as $|\lambda|$ becomes larger.
One may also consider other stochastic processes on $\mathbb{P}$.

We may not need to see a $\WL$-category as a collection of processes in which a space can evolve probabilistically.
For $x, x'\in X$, the \textit{distance} $d(x, x')$ is not yet determined in some sense.
Namely, we can evaluate it at $f\in\Lambda$ to obtain an observed value $d_{f}(x, x'):=d(x, x')(f)\in[0, \infty]$.
If $f$ is expressed as $\sum a_{\lambda}m_{\lambda}$ with $a_{\lambda}\in\mathbb{N}$, the observed value is given by $d_{f}(x, x')=\min\{d_{\lambda}(x, x')\}$.
The partition achieving the minimum depends on $x$, $x'$ and $f$.

As Proposition \ref{complete-metric} shows, $\{(X, d_{n})\}_{n\in\mathbb{N}_{>0}}$ is a sequence of $\mathbb{L}$-categories, which does not depend on the measure.
It might be also fine to say a $\WL$-category X includes a discrete-time sequence of $\mathbb{L}$-categories.
\end{rem}

By replacing the base category $\mathbb{L}$ with $\WL$, we can endow spaces with a structure related to the composition operator on $\Lambda$ in the following sense.
Since $\mathbb{W}$ is a comonad, there is a map $c_{\mathbb{L}}:\WL\to\mathbb{W}(\WL)$, which is induced by the composition operator on $\Lambda$.

Let $(X, d)\in\WLCat$.
Then we have a map 
\[c_{\mathbb{L}}\circ d:X\times X\to\mathbb{W}(\WL),\]
defined for all $x, x'\in X$ by,
\[(c_{\mathbb{L}}\circ d)(x, x'):\Lambda\to\WL,\]
where for each $f\in\Lambda$,
\[((c_{\mathbb{L}}\circ d)(x, x'))(f):\Lambda\to\mathbb{L}, g\mapsto d_{g\circ f}(x, x').\]

Let $|X|$ denote the family $\{(X, d_{f})\}_{f\in\Lambda}$, which carries a natural rig structure with addition $(X, d_{f})\oplus (X, d_{f'}):=(X, d_{f+f'})$ and multiplication $(X, d_{f})\otimes (X, d_{f'}):=(X, d_{f\cdot f'})$.
We let $\operatorname{End}(|X|)$ denote the set of maps from $|X|$ to itself, and equip it with the induced rig structure.
By the definition of the composition operator on $\Lambda$, we have the following.

\begin{prop}Let $(X, d)\in\WLCat$.
Then the action
\[\Lambda\to\operatorname{End}(|X|), \ \ g\mapsto ((X, d_{f})\mapsto (X, d_{g\circ f}))\]
is a rig homomorphism.
\end{prop}

\begin{rem}Let $(X, d)\in\WLCat$ and $x, x'\in X$.
Let also $\lambda, \lambda'\in\mathbb{P}$, and write $m_{\lambda}\circ m_{\lambda'}=\sum_{\mu}a_{\mu}m_{\mu}$.
Then we have $d_{m_{\lambda}\circ m_{\lambda'}}(x, x')=\min_{\mu}\{d_{\mu}(x, x')\}$.

Put the lexicographic order $\leq_{l}$ on $\mathbb{P}_n$ for all $n\in\mathbb{N}_{>0}$, which may not seem very natural, and put the induced order on $\mathbb{P}$.
Namely, for $\lambda, \lambda'\in\mathbb{P}$, $\lambda\leq\lambda'$ if and only if $|\lambda|<|\lambda'|$ or $|\lambda|=|\lambda'|$ and $\lambda\leq_{l}\lambda'$.
Thus we obtain a unique minimum partition achieving the minimum $\min_{\mu}\{d_{\mu}(x, x')\}$.
By using this, the action above induces a map $\mathbb{P}\times\{d_{\lambda}(x, x')\}_{\lambda\in\mathbb{P}}\to\{d_{\lambda}(x, x')\}_{\lambda\in\mathbb{P}}$.

It is not easy to compute $g\circ f$ for general $g, f\in\Lambda$.
However, for any $n, n'\in\mathbb{N}_{>0}$, $m_{(n)}\circ m_{(n')}$ is simply $m_{(nn')}$, which plays a role in the theory of Witt vectors of rings.
This in particular defines an action of $\mathbb{N}_{>0}$ on $\{(X, d_{(n)})\}_{n\in\mathbb{N}_{>0}}$.
\end{rem}

\subsection{Relationships of $\WL$-categories and $\mathbb{L}$-categories}
In the following, we investigate how usual $\mathbb{L}$-categories can be interpreted as $\WL$-categories.

\begin{thm}The map $\theta:\mathbb{L}\to\mathbb{W}(\mathbb{L})$ defined by
\[
\theta(r)(m_{\lambda})=
\begin{cases}
nr & \text{if } \lambda = (n) \text{ for some } n \\
\infty & \text{otherwise}
\end{cases}
\]
is a monoidal functor.
\end{thm}
\begin{proof}We first show that $\theta(r):\Lambda\to\mathbb{L}$ is a rig homomorphism for any $r\in\mathbb{L}$.
Since $\{m_{\lambda}\}_{\lambda\in\mathbb{P}}$ is an $\mathbb{N}$-module basis, $\theta(r)$ preserves addition.
Let $n, n'\in\mathbb{N}_{>0}$.
Then $m_{(n)}m_{(n')}=m_{(n+n')}+m_{(n, n')}$.
So we have
\begin{eqnarray*}
\theta(r)(m_{(n)}m_{(n')})&=&\min\{\theta(r)(m_{(n+n')}), \theta(r)(m_{(n, n')})\}\\
&=&\min\{(n+n')r, \infty\} \\
&=&(n+n')r\\
&=&\theta(r)(m_{(n)})+\theta(r)(m_{(n')})\\
&=&\theta(r)(m_{(n)})\otimes_{\mathbb{L}}\theta(r)(m_{(n')}).
\end{eqnarray*}

Let $\lambda\in\mathbb{P}$ not of the form $(n')$ for any $n'\in\mathbb{N}$, then $m_{(n)}m_{\lambda}=\sum_{\mu}c_{\mu}m_{\mu}$, where any $\mu$ is not of the form $(k)$ with $k\in\mathbb{N}_{>0}$ and $c_{\mu}\in\mathbb{N}$.
So we have
\begin{eqnarray*}
\theta(r)(m_{(n)}m_{\lambda})&=&\min_{\mu}\{\theta(r)(m_{\mu})\}\\
&=&\min_{\mu}\{\infty\} \\
&=&\infty\\
&=&\theta(r)(m_{(n)})\otimes_{\mathbb{L}}\theta(r)(m_{\lambda}).
\end{eqnarray*}

Let $\lambda, \lambda'\in\mathbb{P}$ both of which are not of the form $(n)$ for any $n\in\mathbb{N}_{>0}$.
Then by the same argument above, we have $\theta(r)(m_{\lambda}m_{\lambda'})=\theta(r)(m_{\lambda})\otimes_{\mathbb{L}}\theta(r)(m_{\lambda'})$.

So we have shown that $\theta(r)\in\WL$ for all $r\in[0, \infty]$.

Since $\Delta^{\times}(m_{(n)})=m_{(n)}\otimes m_{(n)}$ for all $n\in\mathbb{N}$, we have
\begin{eqnarray*}
(\theta(r)\otimes_{\WL}\theta(r'))(m_{(n)})&=&\theta(r)(m_{(n)})+\theta(r')(m_{(n)})\\
&=&n(r+r')\\
&=&\theta(r\otimes_{\mathbb{L}}r')(m_{(n)}),
\end{eqnarray*}
for any $r, r'\in\mathbb{L}$.

For any $\lambda\in\mathbb{P}$ which is not of the form $(n)$, $m_{(k)}\otimes m_{(k)}$ with any $k\in\mathbb{N}_{>0}$ does not appear in $\Delta^{\times}(m_{\lambda})$.
So $(\theta(r)\otimes_{\WL}\theta(r'))(m_{\lambda})=\infty=\theta(r\otimes_{\mathbb{L}}r')(m_{\lambda})$ for any $r, r'\in\mathbb{L}$.

By definition, $\theta(0)=\overline{0}$, and $\theta(r)\leq_{\WL}\theta(r')$ if $r\leq_{\mathbb{L}}r'$.
\end{proof}

\begin{rem}Let $r, r'\in\mathbb{L}$, then by definition
\[
\theta(\min\{r, r'\})(m_{\lambda})=
\begin{cases}
n\min\{r, r'\} & \text{if } \lambda = (n) \text{ for some } n \\
\infty & \text{otherwise}.
\end{cases}
\]
On the other hand, for any $n\in\mathbb{N}_{>0}$,
\begin{eqnarray*}
(\theta(r)\oplus_{\WL}\theta(r'))(m_{(n)})&=&\min\{\theta(r)(m_{(n)}), \theta(r')(m_{(n)})\}\\
&=&\min\{nr, nr'\} \\
&=&\theta(\min\{r, r'\})(m_{(n)}),
\end{eqnarray*}
since $\Delta^{+}(m_{(n)})=m_{(n)}\otimes 1+1\otimes m_{(n)}$.

However, by definition, since $\Delta^{+}(m_{(2, 1)})=m_{(2, 1)}\otimes 1+m_{(2)}\otimes m_{(1)}+m_{(1)}\otimes m_{(2)}+1\otimes m_{(2, 1)}$, we have
\[(\theta(r)\oplus_{\WL}\theta(r'))(m_{(2, 1)})=\min\{\infty, 2r+r', r+2r'\}.\]
So, in general, $\theta(\min\{r, r'\})(m_{\lambda})\neq(\theta(r)\oplus_{\WL}\theta(r'))(m_{\lambda})$, namely $\theta$ does not preserve addition and is not a rig homomorphism.

Note that, in contrast to the rig $\mathbb{L}$, the rig $\WL$ is not of characteristic one.
Indeed, $\overline{0}\oplus_{\WL}\overline{0}\neq\overline{0}$, which can be checked directly by $(\overline{0}\oplus_{\WL}\overline{0})(m_{(2,1)})=0$ and $\overline{0}(m_{(2, 1)})=\infty$.

Thus there does not exist a rig homomorphism from the rig $\mathbb{L}$ to the rig $\WL$, but we have a monoidal functor from the monoidal poset $\mathbb{L}$ to the monoidal poset $\WL$.
Note that the construction of the order $\leq_{\mathbb{L}}$ does not work well for rings as we have seen. 

Note also that, for a commutative ring $R$, there is a famous ring isomorphism $\mathbb{W}(R)\cong 1+tR[[t]]$.
Under this ring isomorphism, the addition in $\mathbb{W}(R)$ corresponds to the usual multiplication in $1+tR[[t]]$, and the multiplication in $\mathbb{W}(R)$ corresponds to an intricate multiplication in $1+tR[[t]]$.
In addition, categorical spectra are defined to be the stabilisations of $(\infty, \infty)$-categories instead of $(\infty, 0)$-categories, in which all morphisms are inverted.

\end{rem}

\begin{prop}The map $\tau:\WL\to\mathbb{L}, \tau(f)=f(m_{(1)})$ is a monoidal functor.
\end{prop}
\begin{proof}Since $\Delta^{\times}(m_{(1)})=m_{(1)}\otimes m_{(1)}$, for any $f, g\in\WL$, we have
\begin{eqnarray*}
\tau(f\otimes_{\WL}g)&=&(f\otimes_{\WL}g)(m_{(1)})\\
&=&f(m_{(1)})+g(m_{(1)})\\
&=&\tau(f)\otimes_{\mathbb{L}}\tau(g).
\end{eqnarray*}
Again by definition, $\tau(\overline{0})=0$, and $\tau(f)\leq_{\mathbb{L}}\tau(g)$ if $f\leq_{\WL}g$.
\end{proof}

These functors $\theta$ and $\tau$ induces functors again denoted by $\theta:\mathbb{L}\mathchar`-\mathbf{Cat}\to\WL\mathchar`-\mathbf{Cat}$ and $\tau:\WL\mathchar`-\mathbf{Cat}\to\mathbb{L}\mathchar`-\mathbf{Cat}$ respectively.

By construction, we have the following.

\begin{prop}$\theta:\mathbb{L}\mathchar`-\mathbf{Cat}\to\WL\mathchar`-\mathbf{Cat}$ is fully faithful.
\end{prop}

For $(X, d)\in\mathbb{L}\mathchar`-\mathbf{Cat}$ we denote $\theta(X, d)=(X, \tilde{d})\in\WL\mathchar`-\mathbf{Cat}$.
As we saw above, for any $n\in\mathbb{N}_{>0}$ we have $(X, \tilde{d}_{(n)})\in\mathbb{L}\mathchar`-\mathbf{Cat}$ such that for any $x, x'\in X$, $\tilde{d}_{(n)}(x, x')=nd(x, x')$.
Thus we may obtain a sequence $\{(X, \tilde{d}_{(n)})\}_{n\in\mathbb{N}_{>0}}$ of $\mathbb{L}$-categories, linearly and uniformly expanding over discrete time, with the initial state $(X, \tilde{d}_{(1)})=(X, d)$.

Similarly, for any $\lambda\in\mathbb{P}$ which is not of the form $(n)$ for any $n\in\mathbb{N}_{>0}$, and for any $x, x'\in X$, we have $\tilde{d}_{\lambda}(x, x')=\infty$.
Note that $(X, \tilde{d}_{\lambda})$ is no longer an $\mathbb{L}$-category, but it is a kind of discrete space.
If we consider the Plancherel growth process, the probability that a randomly chosen $\lambda$ is of the form $(|\lambda|)$ will tend to decrease as $|\lambda|$ becomes larger.
So it becomes increasingly difficult for $(X, \tilde{d}_{\lambda})$ to be an $\mathbb{L}$-category as $|\lambda|$ becomes larger.
Note that, if a partition $\lambda$ is not of the form $(|\lambda|)$, then $(X, \tilde{d}_{\lambda})$ is not an $\mathbb{L}$-category and the subsequent one also can not be an $\mathbb{L}$-category, since the transition probability from $(\lambda_1, \lambda_2, \dots, \lambda_k)$ with $\lambda_2\neq 0$ to $(|\lambda|+1)$ is $0$.

Note also that in this case $\tilde{d}_{n}=\tilde{d}_{(n)}$ for all $n\in\mathbb{N}_{>0}$.
So we can not distinguish $\{(X, d_{(n)})\}_{n\in\mathbb{N}_{>0}}$ and  the non-stochastic sequence $\{(X, \tilde{d}_{n})\}_{n\in\mathbb{N}_{>0}}$.

Conversely let $(Y, d)\in\WL$.
By taking $\tau:\WL\mathchar`-\mathbf{Cat}\to\mathbb{L}\mathchar`-\mathbf{Cat}$, we have its initial-state $\mathbb{L}$-category $(Y, d_{(1)})$.

\begin{definition}We define a subset $\WL_{l}$ of $\WL$ to be
\[\{f\in\WL|nf(m_{(1)})\leq_{\mathbb{L}}f(m_{(n)}), \forall n\in\mathbb{N}_{>0}\}.\]
\end{definition}

\begin{prop}The monoidal category structure on $\WL$ restricts to $\WL_{l}$ and the monoidal functor $\tau:\WL\to\mathbb{L}$ restricts to the monoidal functor $\tau:\WL_{l}\to\mathbb{L}$.
\end{prop}

\begin{proof}We need to show the first claim.
It is clear that the partially order structure, hence the category structure, on $\WL$ restricts to $\WL_{l}$.
By definition, $\overline{0}\in\WL_{l}$.
Let $f, g\in\WL_{l}$.
Then by definition, for all $n\in\mathbb{N}_{>0}$,
\begin{eqnarray*}
(f\otimes_{\mathbb{W}(\mathbb{L})}g)(m_{(n)})&=&f(m_{(n)})+g(m_{(n)})\\
&\geq_{\mathbb{L}}&n(f(m_{(1)})+g(m_{(1)}))\\
&=&n(f\otimes_{\mathbb{W}(\mathbb{L})}g)(m_{(1)}).
\end{eqnarray*}
\end{proof}
Also by definition we have the following.
\begin{prop}The monoidal functor $\theta:\mathbb{L}\to\WL$ factors through the monoidal functor $\theta:\mathbb{L}\to\WL_{l}$.
\end{prop}

\begin{prop}The monoidal functor $\theta:\mathbb{L}\to\WL_{l}$ is left adjoint to $\tau:\WL_{l}\to\mathbb{L}$.
\end{prop}
\begin{proof}
Let $r\in\mathbb{L}$ and $f\in\WL_{l}$.
Suppose $\theta(r)\leq_{\WL}f$.
Then by definition for any $\lambda\in\mathbb{P}$, we have $\theta(r)(m_{\lambda})\leq_{\mathbb{L}}f(m_{\lambda})$.
Again by definition,
\[
\theta(r)(m_{\lambda})=
\begin{cases}
nr & \text{if } \lambda = (n) \text{ for some } n \\
\infty & \text{otherwise.}
\end{cases}
\]
So we have $r\leq_{\mathbb{L}}f(m_{(1)})=\tau(f)$.

Conversely, suppose $r\leq_{\mathbb{L}}\tau(f)$.
Since $\tau(f)=f(m_{(1)})$, for all $n\in\mathbb{N}_{>0}$, we have $nr\leq_{\mathbb{L}}nf(m_{(1)})$.
Since $f\in\WL_{l}$, we have $nr\leq_{\mathbb{L}}f(m_{(n)})$.
By the definition of $\leq_{\WL}$ and that of $\theta(r)$, we have $\theta(r)\leq_{\WL}f$.
\end{proof}

\begin{rem}Let $X$ be a category enriched over $\WL_{l}$.
Then by definition, for all $x, x'\in X$ and $n\in\mathbb{N}_{>0}$, we have $nd_{(1)}(x, x')\geq d_{(n)}(x, x')$.
In other words, the identity map on $X$ gives an $n$-Lipschitz map $(X, d_{(1)})\to(X, d_{(n)})$.
\end{rem}

The Lawvere quantale $\mathbb{L}$ is an example of (commutative) quantales.
One might consider the Witt vectors in other quantales as well.
Also, $\Lambda$ is an example of composition algebras, so one might choose other composition algebras.
It would be worth emphasizing that the combinatorial or probabilistic aspect of $\mathbb{P}$ has contributed to the theory of Witt vectors.

\end{document}